\documentclass[11pt]{amsart}
\usepackage[usenames]{color}
\usepackage{fullpage,amscd,amssymb,latexsym,url,hyperref}
\usepackage[all,2cell,ps]{xy}
\usepackage[T1]{fontenc}

\theoremstyle{plain}
\newtheorem{thm}{Theorem}[section]
\newtheorem{theorem}[thm]{Theorem}

\newtheorem{observation}[thm]{Observation}
\newtheorem{corollary}[thm]{Corollary}

\newtheorem{lemma}[thm]{Lemma}

\newtheorem{prop}[thm]{Proposition}
\newtheorem{proposition}[thm]{Proposition}

\theoremstyle{definition}
\newtheorem{de}[thm]{Definition}

\newtheorem{remark}[thm]{Remark}
\newtheorem{example}[thm]{Example}
\newtheorem{algorithm}[thm]{Algorithm}

\newcommand{\Dis}{\rm Dis}
\newcommand{\Z}{\mathbb{Z}}

\newcommand{\A}{\mathcal{A}}

\newcommand{\bigo}{\mathcal{O}}

\newcommand{\Aff}{\mathrm{Aff}}
\newcommand{\ext}[1]{\mathrm{Ext}(#1)}
\newcommand{\aff}[1]{\mathrm{Aff}(#1)}
\newcommand{\aut}[1]{\mathrm{Aut}(#1)}
\newcommand{\dis}[1]{\mathrm{Dis}(#1)}
\newcommand{\End}[1]{\mathrm{End}(#1)}
\newcommand{\lmlt}[1]{\mathrm{LMlt}(#1)}

\newcommand{\im}[1]{\mathrm{Im}(#1)}
\newcommand{\Ker}[1]{\mathrm{Ker}(#1)}
\newcommand{\ld}{\backslash}

\numberwithin{equation}{section}

\begin{document}

\title{Subquandles of affine quandles}

\author{P\v remysl Jedli\v cka}
\author{Agata Pilitowska}
\author{David Stanovsk\'y}
\author{Anna Zamojska-Dzienio}

\address{(P.J.) Department of Mathematics, Faculty of Engineering, Czech University of Life Sciences, Kam\'yck\'a 129, 16521 Praha 6, Czech Republic}
\address{(A.P., A.Z.) Faculty of Mathematics and Information Science, Warsaw University of Technology, Koszykowa 75, 00-662 Warsaw, Poland}
\address{(D.S.) Department of Algebra, Faculty of Mathematics and Physics, Charles University, Sokolovsk\'a 83, 18675 Praha 8, Czech Republic}

\email{(P.J.) jedlickap@tf.czu.cz}
\email{(A.P.) apili@mini.pw.edu.pl}
\email{(D.S.) stanovsk@karlin.mff.cuni.cz}
\email{(A.Z.) A.Zamojska-Dzienio@mini.pw.edu.pl}

\thanks{Our joint research started within the framework of the Czech-Polish cooperation grants 7AMB13PL013 and 8829/R13/R14. The second and the fourth authors were supported by the statutory grant of the Warsaw University of Technology 504/02476/1120. The third author was partly supported by the Czech Science Foundation grant 13-01832S}


\date{\today}

\begin{abstract}
A quandle will be called quasi-affine, if it embeds into an affine quandle.
Our main result is a characterization of quasi-affine quandles, by group-theoretic properties of their displacement group, by a universal algebraic condition coming from the commutator theory, and by an explicit construction over abelian groups. As a consequence, we obtain efficient algorithms for recognizing affine and quasi-affine quandles, and we enumerate small quasi-affine quandles. We also prove that the ``abelian implies quasi-affine" problem of universal algebra has affirmative answer for the class of quandles.
\end{abstract}

\maketitle

\section{Introduction}

Affine quandles (also called Alexander quandles) play a prominent role in quandle theory, both from the algebraic perspective \cite{AG,ESG,Hou-aut,Hou,HSV,JPSZ}, and in applications in knot theory, due to a close connection between affine colorings and the Alexander invariant \cite{Bae,CDP,Joy}. 
In the present paper, we look at the structure and abstract properties of quandles that embed into affine quandles, that is, that are isomorphic to a subquandle of an affine quandle. We will call such quandles \emph{quasi-affine}. For example, free medial quandles are quasi-affine, but not affine \cite{JPZ}.

At the moment, our motivation is purely algebraic, with emphasis on computational aspects. What is their structure? How many are there? How to recognize them? 
We will present both a structural theorem, and a computationally feasible characterization of quasi-affine quandles (Theorem \ref{thm:quasi-affine}).
The former goal is achieved using a special kind of central extension (Definition \ref{def:ext}). Together with a convenient isomorphism theorem (Theorem \ref{thm:iso}), this allows fairly efficient enumeration (Section \ref{sec:enum}). We also present polynomial-time algorithms (subquadratic with respect to the input size) for recognition of affine and quasi-affine quandles (Algorithms \ref{alg:affine} and \ref{alg:quasi-affine}). The key property behind the results is abelianness and semiregularity of the displacement group.

A secondary motivation for our study comes from universal algebra. 
One of the major projects in universal algebra is to determine abstract conditions under which a general algebraic structure embeds into an affine one; formally, when it is a subreduct of a module \cite{SS,Sz}. Such algebras are also called quasi-affine.
In particular, a longstanding open problem asks, whether every idempotent algebraic structure satisfying certain syntactic condition called \emph{abelianness} is quasi-affine \cite{K}. We confirm the conjecture for the class of quandles. As far as we know, after \cite{Sta}, this is only the second result when the problem is confirmed for a broad class of idempotent algebras failing every non-trivial Mal'tsev condition. 

Our results are also interesting in the context of the theory of \emph{modes} \cite{RS}, which develops its own theory of linear representations (medial quandles are examples of modes).

As a byproduct, we prove several new results for affine quandles, complementing existing theory \cite{Holmes, Hou}. Our main tool, Theorem \ref{thm:affine}, characterizes affine quandles in a way similar to Theorem \ref{thm:quasi-affine}, and is of independent interest. In particular, our characterization of the displacement groups results in an algorithm for recognition of affine quandles which is a tremendous improvement over the brute-force method of \cite{MNT}. 

Affine quandles are \emph{medial}, and so are their subquandles. We could therefore build upon the structure theory developed in \cite{JPSZ}, where we represented medial quandles by certain heterogeneous affine structure. However, it turned out that our theory was easier to develop from scratch, because meshes of quasi-affine quandles are very symmetric and thus better viewed as extensions. We will use the results of \cite{JPSZ} in the final sections: the isomorphism theorem for quasi-affine quandles will be proved by specializing the (more general) isomorphism theorem for affine meshes.

The paper is organized as follows.
In Section \ref{sec:terminology}, we recall some of the quandle theory needed in the paper, and we formulate our main results, Theorems \ref{thm:quasi-affine} and \ref{thm:affine}.
Section \ref{sec:hs} contains an auxiliary module-theoretic result, called the \emph{Hou-\v S\v tov\'\i\v cek extension lemma}. We see it as \emph{the} module-theoretic principle behind the structure theory of affine quandles. 
In Section \ref{sec:ext}, we introduce semiregular extensions, a concept used to represent quasi-affine quandles, and prove a few elementary properties.
Section \ref{sec:ua} relates the universal algebraic and quandle theoretic principles.
In Section \ref{sec:proofs}, we prove Theorems \ref{thm:quasi-affine} and \ref{thm:affine}. In the finite case, the somewhat cumbersome characterization of affine quandles can be replaced by a more esthetic condition; this is the subject of Theorem \ref{thm:affine_fin}. We also include several interesting examples and counterexamples related to the abstract conditions that appear throughout the paper.
Section \ref{sec:alg} contains the recognition algorithms based on our characterizing theorems, including their complexity analysis.
In Section \ref{sec:iso}, we relate semiregular extensions to the affine meshes of \cite{JPSZ}, and prove the isomorphism theorem for semiregular extensions.
In the last section, we apply it to enumerate small quasi-affine quandles. 

The paper is aimed at both quandlists and universal algebraists. The proof of the main theorems does not rely on the universal algebraic concepts, and thus Section \ref{sec:ua} could be safely skipped. Nevertheless, we advise to take this interesting abstraction into account. 
A universal algebraic background can be learnt from \cite{Ber}.
Our approach to quandles is best summarized in the introductory parts of \cite{HSV}.
A comprehensive study of affine quandles can be found in \cite{Hou-aut,Hou}, an alternative approach was developed by Holmes in her Master's thesis \cite{Holmes}. The present work was influenced by some of her ideas presented in the thesis.

%This crucial fact takes inspiration in \cite{H}; we are grateful to Jan \v Stov\'\i\v cek who pointed us to the principle behind the lemma, allowing us to drop the finiteness assumptions.

%One could also ask what are homomorphic images of affine quandles. This turns out to be an interesting non-trivial problem that will be addressed in a sequel paper.

\section{Terminology and main results}\label{sec:terminology}

\subsection{Quandles}

All unproved results stated in this section can be found in the introductory part of \cite{HSV} (using the present notation), and most of them also elsewhere (often in a different notation).

We will write mappings acting on the left, hence conjugation in groups will be denoted by $x^y=yxy^{-1}$, and consequently, the commutator will be defined as $[x,y]=y^xy^{-1}=xyx^{-1}y^{-1}$.

A \emph{quandle} is an algebraic structure $Q=(Q,*)$ which is \emph{idempotent} (it satisfies the identity $x*x=x$), \emph{uniquely left divisible} (for every $x,y$, there is a unique $z$ such that $x*z=y$, to be denoted $z=x\ld y$), and \emph{left distributive} (it satisfies the identity $x*(y*z)=(x*y)*(x*z)$).
The mappings $L_x:Q\to Q$, $L_x(y)=x*y$, will be called \emph{left translations}. It follows from the quandle axioms that all left translations are automorphisms of $Q$. We will often drop the adjective ``left". For universal-algebraic purposes, we will regard left division as a basic operation, i.e., it can be used in terms.

%The quandle operation will be implicitly denoted multiplicatively. In certain constructions, other symbols will be used, e.g., $*$ will be used for affine quandles, or more generally, for semiregular extensions. Addition will be reserved for abelian groups.

Two important permutation groups are associated to every quandle:
the (left) \emph{multiplication group}, generated by all (left) translations,
\[\lmlt Q=\langle L_a:\ a\in Q\rangle\leq\aut Q,\]
and its subgroup, the \emph{displacement group}, defined by
\[\dis Q=\langle L_aL_b^{-1}:\ a,b\in Q\rangle\leq\lmlt Q.\]
%In any quandle, $\dis Q=\{L_{a_1}^{k_1}\ldots L_{a_n}^{k_n}:\sum k_i=0, a_i\in Q\}$.
%It follows that $\dis Q=\{L_{a_1}^{k_1}\dots L_{a_n}^{k_n}:\ a_1,\dots,a_n\in Q\text{ and }\sum_{i=1}^n k_i=0\}.$
Both groups have the same orbits of the natural action on $Q$, to be called \emph{orbits} of the quandle $Q$, and denoted 
\[Qe=\{\alpha(e):\ \alpha\in\lmlt Q\}=\{\alpha(e):\ \alpha\in\dis Q\}.\]
Orbits are subquandles of $Q$. They form a block system, to be called the \emph{orbit decomposition} of $Q$.

A quandle is called \emph{medial} (also \emph{entropic} or \emph{abelian} elsewhere) if it satisfies the identity $(x*y)*(u*v)=(x*u)*(y*v)$. 
This is equivalent to abelianness of the displacement group. 
Therefore, if $Q$ is medial, $\alpha\in\dis Q$ and $x,y\in Q$, we have $\alpha^{L_y^{-1}L_x}=\alpha$, and thus
\begin{equation}\label{Eq:conj}
\alpha^{L_x}=\alpha^{L_y}.
\end{equation}
%In medial quandles, $(xy)\ld(uv)=(x\ld u)(y\ld v)$ also holds. 

Let $(A,+)$ be an abelian group, $f$ its automorphism, and define an operation on the set $A$ by
\[a*b=(1-f)(a)+f(b).\]
Then $(A,*)$ is a medial quandle, to be denoted $\aff{A,f}$, and called \emph{affine} over the group $(A,+)$.
Here $1$ refers to the identity mappings, hence $g=1-f$ is the mapping $g(x)=x-f(x)$.
%Affine quandles are related to $\Z[t,t^{-1}]$-modules.
A product of two affine quandles  $Q=\aff{A,f}$ and $R=\aff{B,g}$
is affine since $Q\times R=\aff{A\times B,f\times g}$.
Affine quandles with $f=1$ will be called \emph{projection quandles} (also \emph{trivial quandles} elsewhere), since the operation is the right projection, $a*b=b$.
The projection quandle of size~$k$ (possibly infinite) will be denoted by $\mathrm{Proj}(k)$.

Following the universal algebraic terminology, quandles embeddable into affine quandles will be called shortly \emph{quasi-affine}. 
(Contrary to universal algebra, quandle-theoretic definition of affineness is weaker. In universal algebra, an algebraic structure is called affine if and only if it is polynomially equivalent to a module; affine quandles are only assumed to be reducts of modules.)

We will say that $\dis Q$ is \emph{tiny} if $\dis{Q}=\{L_xL_e^{-1}:\ x\in Q\}$ for some $e\in Q$.
Affine quandles have tiny displacement groups (the converse is not true): for $Q=\aff{A,f}$ we have 
\[\dis Q=\{x\mapsto a+x:\ a\in\im{1-f}\}=\{L_xL_0^{-1}:x\in Q\},\] since $L_aL_b^{-1}(x)=(1-f)(a-b)+x$.
Hence $\dis Q\simeq\im{1-f}$, and the orbits of $Q$ are the cosets of $\im{1-f}$. 

%The Cayley kernel consists of those pairs $(a,b)$ such that $(1-f)(a)=(1-f)(b)$, i.e., such that $a-b\in\mathrm{Ker}(1-f)$. In other words, the blocks of the Cayley kernel are the cosests of $\Ker{1-f}$.

%The \emph{Cayley representation} is the mapping $\lambda:Q\to\S(Q)$, $x\mapsto L_x$. Analogously to groups, $\lambda$ is a quandle homomorphism (with respect to the conjugation operation on $\S(Q)$, the symmetric group over $Q$), but, unlike for groups, $\lambda$ is not necessarily one-on-one. Its kernel, $\ker(\lambda)=\{(x,y):\lambda(x)=\lambda(y)\}$, will be called the \emph{Cayley kernel}. Quandles with trivial Cayley kernel are often called \emph{faithful}.

\subsection{Multitransversals}

Informally, a \emph{multiset} is a generalization of the notion of a set where elements can repeat. Tuples can be turned into multisets, forgetting the indexing. Multisets will be denoted by double brackets $\{\{...\}\}$.

A multitransversal for a block system is a multiset which takes from each block the same amount of elements, i.e., a multiset $T$ such that $|T\cap B_1|=|T\cap B_2|$ for every pair of blocks $B_1,B_2$; the size of the intersection will be called the \emph{multiplicity} of $T$. If $G$ is a group and $H$ its subgroup, then by a (left) multitransversal of $G/H$ we mean a multitransversal of the block system $\{a+H:a\in G\}$.

\begin{lemma}\label{l:mtr2}
Let $G$ be a group, $\varphi\in\End G$, and let $T$ be a left transversal of $G/\im\varphi$. Then $\varphi(T)$, as a multiset, is a left multitransversal of~$\im\varphi/\im{\varphi^2}$.
The multiplicity of~$\varphi(T)$ is equal to $|\Ker\varphi/\Ker\varphi\cap\im\varphi|$.
\end{lemma}

\begin{proof}
Let $t,s\in T$. We have $\varphi(t)\varphi(s)^{-1}\in\im{\varphi^2}$ if and only if $\varphi(ts^{-1})=\varphi^2(a)$ for some $a\in G$.
Now $\varphi(ts^{-1})=\varphi^2(a)$ if and only if $\varphi(ts^{-1}\varphi(a)^{-1})=1$, that is, if and only if $ts^{-1}\varphi(a)^{-1}\in\Ker \varphi$. Consequently,
$\varphi(t)\varphi(s)^{-1}\in\im{\varphi^2}$ if and only if $ts^{-1}\in\Ker\varphi\cdot \im\varphi$.
Each block of $G/(\Ker \varphi\cdot\im \varphi)$ contains the same amount of blocks of $G/\im\varphi$, and thus the same amount of elements of $T$. 
Looking at the block $\Ker \varphi\cdot\im \varphi$, we see that the multiplicity is
$|(\Ker\varphi\cdot\im\varphi)/\im\varphi|$.
By the second isomorphism theorem, this is equal to $|\Ker\varphi/\Ker\varphi\cap\im\varphi|$.
\end{proof}

\subsection{Main results}

Now we can formulate the main results, the characterization theorems for affine and quasi-affine quandles.

Recall that a permutation group $G$ acting on a set $X$ is called {\em semiregular} (the terms \emph{free} or \emph{fixpoint-free} are also used in literature) if non-trivial permutations from~$G$ are regular, i.e., have no fixed points. In other words, if $g(x)\neq x$ for every $1\neq g\in G$ and $x\in X$.

The \emph{semiregular extension}, $\ext{A,f,\bar d}$, will be defined in Section \ref{sec:ext}. It is a particular type of a central extension of a projection quandle over the affine quadle $\aff{A,f}$ (for more information on centrality see Remark \ref{rem:central}).

In condition (4), abelianness refers to a certain syntactic condition, to be explained in Section \ref{sec:ua}. It is a generalization of the idea that a group $G$ is abelian if and only if the diagonal of $G^2$ forms a normal subgroup.

\begin{theorem}\label{thm:quasi-affine}
The following statements are equivalent for a quandle $Q$:
\begin{enumerate}
	\item $Q$ is quasi-affine;
	\item $\dis Q$ is abelian and semiregular;
	\item $Q$ is isomorphic to $\ext{A,f,\bar d}$ for some abelian group~$A$, its automorphism $f$ and some tuple $\bar d=(d_i:i\in I)$ of elements of $A$ (the extension can be taken indecomposable);
	\item $Q$ is abelian (in the sense of \cite{FM}).
\end{enumerate}
\end{theorem}

\begin{theorem}\label{thm:affine}
The following statements are equivalent for a quandle $Q$:
\begin{enumerate}
	\item $Q$ is affine;
	\item $\dis Q$ is abelian, semiregular and the multiset $\{\{L_aL_e^{-1}:a\in T\}\}$ is balanced for every $e\in Q$ and every transversal $T$ of the orbit decomposition. 
	\item[(2')] $\dis Q$ is abelian, semiregular and the multiset $\{\{L_aL_e^{-1}:a\in T\}\}$ is balanced for some $e\in Q$ and some transversal $T$ of the orbit decomposition. 
	\item[(3)] $Q$ is isomorphic to $\ext{A,f,\bar d}$ for some abelian group~$A$, its automorphism $f$ and some balanced tuple $\bar d=(d_i:i\in I)$ of elements of $A$ (the extension can be taken indecomposable).
\end{enumerate}
\end{theorem}

A multiset $\{\{L_aL_e^{-1}:a\in T\}\}$ is called \emph{balanced} if it is a multitransversal of $\dis Q/[\dis Q,L_e]$.
A tuple $\bar d$ is called \emph{balanced}, if it is a multitransversal of $A/\im{1-f}$.
As we shall see in Theorem \ref{thm:affine_fin}, if $Q$ is finite, than these two balancedness conditions are equivalent to the fact that the multiplication table of $Q$ is balanced in a particularly nice way.

Proving the implications $(1)\Rightarrow(2)$ is fairly straightforward, and the semiregular extensions are designed in a way that the implications $(2)\Rightarrow(3)$ also prove smoothly. The real work is proving the implications $(3)\Rightarrow(1)$.
In either case, we are given a particular semiregular extension $Q$, and we need to find an affine representation, that is, a concrete group $A$ and its automorphism $f$ such that $Q$ is isomorphic to, resp. embeds into, $\aff{A,f}$. This is not as easy as one might expect. One of the difficulties is that the group $A$ is not determined uniquely, not even in the affine case. Our method relies on the Hou-\v S\v tov\'\i\v cek extension lemma, which provides a suitable group in the affine case. Therefore, we first prove Theorem \ref{thm:affine}, and then obtain Theorem \ref{thm:quasi-affine} as a corollary.

Our proof of the Hou-\v S\v tov\'\i\v cek extension lemma is not constructive: the abelian group is proved to exist, but no concrete description is given. 
At the moment, we do not know an explicit construction of the affine representation. This has an algorithmic consequence: we are able to check efficiently whether a given multiplication table defines an affine (or quasi-affine) quandle, but we do not know an efficient way to determine the actual group and automorphism.

The universal algebraic perspective also suggests that calculating an explicit affine representation might be difficult: it was so for many of the general results.
For example, the quasi-affine representation proved in \cite{Sta} for differential modes is also non-constructive, using the indirect syntactic method of \cite{SS}.

\section{The Hou-\v S\v tov\'\i\v cek extension lemma}\label{sec:hs}

\begin{lemma}[Hou-\v S\v tov\'\i\v cek extension lemma]\label{l:hs}
Let $A$ be an abelian group and $\varphi$ its endomorphism. Then there exist an abelian group $E\geq A$ and an epimorphism $\psi:E\to A$ such that $\psi|_A=\varphi$ and $\psi/A:E/A\simeq A/\im{\varphi}$.
\end{lemma}

Here $\psi/A$ is defined by $x+A\mapsto\psi(x)+\im\varphi$. It is a well defined homomorphism, because $x-y\in A$ if and only if $\psi(x)-\psi(y)\in\im\varphi$.

A similar statement was originally proved by Hou \cite[Theorem 4.2]{Hou} under the assumption that $A$ is finite, and used in his enumeration of small affine quandles (or $\Z[t,t^{-1}]$-modules, from his perspective). Later, it found use in Holmes' alternative approach to affine quandles \cite{Holmes}, and in the present paper, it serves as \emph{the} underlying result behind finding the affine representation in the proof of the main theorems.

\v S\v tov\'\i\v cek deserves credit for pointing out that the statement is actually a special case of a classical result in homological algebra, a characterization of hereditary rings in terms of the $\mathrm{Ext^2}$ functor, and that it holds without the finiteness assumption \cite{Sto}.
In fact, the statement is true for modules over any hereditary ring, not just for $\Z$-modules. A ring $R$ is called (left) \emph{hereditary} if submodules of projective (left) $R$-modules are projective, or equivalently, if $\mathrm{Ext}^2(A,B)=0$ for every pair of (left) $R$-modules $A,B$ \cite[Chapter 8]{Rot}. The ring of integers is hereditary, because subgroups of free abelian groups are free. 
%There are numerous characterizations of hereditarity. We shall need the following condition which actually follows easily from the fact that the Ext$^2$ functor vanishes: a ring is hereditary if and only if ......... \cite{...}

\begin{lemma}[Hou-\v S\v tov\'\i\v cek extension lemma, a general version]\label{l:hs_gen}
Let $R$ be a hereditary ring, $A$ an $R$-module and $\varphi$ its endomorphism. Then there exist an $R$-module $E\geq A$ and an epimorphism $\psi:E\to A$ such that $\psi|_A=\varphi$ and $\psi/A:E/A\simeq A/\im{\varphi}$.
\end{lemma}

\begin{proof}
Consider an arbitrary short exact sequence 
$$\begin{CD}
0 @>>> K @>>> E @>\varphi>> I @>>> 0.
\end{CD}$$
%\[0\rightarrow K\rightarrow E\overset{\varphi}\rightarrow I\rightarrow 0\] 
Applying the $\mathrm{Hom}(X,-)$ functor, we obtain an exact sequence
$$\begin{CD}
... @>>> \mathrm{Ext}^1(X,E) @>\mathrm{Hom}(X,\varphi)>> \mathrm{Ext}^1(X,I) @>>> \mathrm{Ext}^2(X,K) @>>> ...
\end{CD}$$
(see \cite[Corollary 6.46]{Rot}). Over a hereditary ring, $\mathrm{Ext}^2(X,K)=0$, and thus $\mathrm{Hom}(X,\varphi)$ is onto.

Applying the general idea to the exact sequence
$$\begin{CD}
0 @>>> \Ker\varphi @>>> A @>\varphi>> \im\varphi @>>> 0
\end{CD}$$
and $X=A/\im\varphi$, we obtain that $\mathrm{Ext}^1(A/\im\varphi,\varphi)$ maps the group $\mathrm{Ext}^1(A/\im\varphi,A)$ onto the group $\mathrm{Ext}^1(A/\im\varphi,\im\varphi)$.
Considering a preimage of the exact sequence 
$$\begin{CD}
0 @>>> \im\varphi @>>> A @>\pi>> A/\im\varphi @>>> 0,
\end{CD}$$
we obtain a module $E$ and homomorphisms $i,\rho,\psi$ such that 
$$\begin{CD}
0 @>>> A @>i>> E @>\rho>> A/\im\varphi @>>>0 \\ 
@. @VV{\varphi}V @VV{\psi}V @|\\
0 @>>> \im\varphi @>{\subseteq}>> A @>\pi>> A/\im\varphi @>>> 0
\end{CD}$$
is a commutative diagram where the left square is a pushout \cite[Lemma 7.28]{Rot}.
Since $i$ is injective (the sequence is exact), we can assume it is an inclusion. Since $\varphi$ is surjective, so is $\psi$, as in any pushout. From the left square, we see that $\psi i=\varphi$, i.e., $\psi|_A=\varphi$. From the right square, we see that $\rho(x)=\pi\psi(x)=\psi(x)+\im\varphi$, and since $\Ker\rho=A$, we obtain that $\psi/A:E/A\simeq A/\im{\varphi}$.
%Consider the exact sequence \[0\rightarrow\Ker\varphi\overset\subseteq\rightarrow A\overset{\varphi}\rightarrow A\rightarrow A/\im\varphi\rightarrow 0.\]
\end{proof}

\section{Semiregular extensions}\label{sec:ext}

\begin{de}\label{def:ext}
Let $A$ be an abelian group, $f$ an automorphism of $A$, $I$ a non-empty set and $d_i\in A$ for $i\in I$. Define an operation on the set $I\times A$ by
\[(i,a)*(j,b)=(j,(1-f)(a)+f(b)+d_i-d_j).\]
It is straightforward to check that the resulting structure $(I\times A,*)$ is a quandle, with
\[(i,a)\ld(j,b)=(j,(1-f^{-1})(a)+f^{-1}(b-d_i+d_j)).\]
The projection $\pi:I\times A\to I$ is a quandle homomorphism onto the projection quandle over $I$, and the fibres of $\pi$,
as subquandles, are all isomorphic to the affine quandle $\aff{A,f}$. We will denote the quandle $(I\times A,*)$ by $\ext{A,f,\bar d}$, where $\bar d=(d_i:i\in I)$, and call it a \emph{semiregular extension over $\aff{A,f}$}.
\end{de}

The name is justified by the following lemma.

\begin{lemma}\label{l:ext1}
Let $A$ be an abelian group, $f$ an automorphism of $A$, $\bar d=(d_i:i\in I)$ a tuple of elements from $A$, and let $Q=\ext{A,f,\bar d}$. Then $\dis Q$ is an abelian semiregular group.
\end{lemma}

\begin{proof}
It is straightforward to calculate that
\[L_{(i,a)}L_{(j,b)}^{-1}(k,c)=(k,c+(1-f)(a-b)+d_i-d_j)=(k,c+t_{i,j,a,b}),\]
where $t_{i,j,a,b}\in A$ is an element of $A$ independent of $k,c$. We see that all generators of $\dis Q$ act as translations over the abelian group $A$, therefore, $\dis Q$ is commutative and semiregular.
\end{proof}

A converse also holds: every quandle with an abelian and semiregular displacement group admits a representation as a semiregular extension. An extension $E=\ext{A,f,\bar d}$ is called \emph{indecomposable}, if the fibers $\{i\}\times A$ are the orbits of $E$.

\begin{lemma}\label{l:ext2}
Let $Q$ be a medial quandle. Let \[E=\ext{\dis Q,f,(L_xL_e^{-1}:x\in T)},\] where
$e\in Q$, $T$ is a transversal of the orbit decomposition of $Q$, and $f$ is the automorphism of $\dis Q$ defined by $f(\alpha)=\alpha^{L_e}$. Then $E$ maps homomorphically onto $Q$, and if $\dis Q$ is semiregular, then $E$ is indecomposable and isomorphic to $Q$.
\end{lemma}

\begin{proof}
We have $E=T\times\dis Q$.
Define a mapping \[\Phi:E\to Q,\quad (x,\alpha)\mapsto \alpha(x).\]
It is onto $Q$, because $T$ is a transversal of the action of $\dis Q$ on $Q$.
We prove that the mapping $\Phi$ is a~homomorphism:
\begin{align*}
 \Phi((x,\alpha)*(y,\beta))&=\Phi(y,(1-f)(\alpha)+f(\beta)+d_x-d_y)\\
 &=(\alpha (\alpha^{-1})^{L_e})\beta^{L_e}(L_xL_e^{-1})(L_yL_e^{-1})^{-1}(y)\\
 &=\alpha (\alpha^{-1}\beta)^{L_e} L_x (y)\stackrel{\scriptsize(\ref{Eq:conj})}=\alpha (\alpha^{-1}\beta)^{L_x} L_x (y)\\
 &=\alpha(x*\alpha^{-1}\beta(y))=\alpha(x)*\beta(y)=\Phi(x,\alpha)*\Phi(y,\beta).
\end{align*}

Now, $\Phi(x,\alpha)=\Phi(y,\beta)$, i.e., $\alpha(x)=\beta(y)$, if and only if $\beta^{-1}\alpha(x)=y$, which can only happen if $x,y$ are in the same orbit, and since $x,y$ were chosen from a transversal, it can only happen if $x=y$. 
So, we have $\alpha(x)=\beta(y)$ if and only if $x=y$ is a fixed point of $\beta^{-1}\alpha$. Consequently, if $\dis Q$ is semiregular, the mapping $\Phi$ is one-to-one.

Indecomposability follows from the fact that the orbits in $Q$ are $Qx=\{\alpha(x):\alpha\in\dis Q\}$, and their preimages under $\Phi$ are the fibers of $E$.
\end{proof}

Note that the two lemmas establish the equivalence $(2)\Leftrightarrow(3)$ of Theorem \ref{thm:quasi-affine}.

\begin{example}\label{ex:affine_ext}
Consider an affine quandle $Q=\aff{A,f}$. Then
\begin{itemize}
	\item $Q=\ext{A,f,(0)}$, hence $\dis Q$ is abelian and semiregular according to Lemma \ref{l:ext1}.
	\item $Q\simeq\ext{\im{1-f},f,((1-f)(a):a\in T)}$, where $T$ is a transversal of the orbit decomposition, as follows from Lemma \ref{l:ext2} under the isomorphism $\im{1-f}\simeq\dis Q$. This extension is indecomposable, the fiber $\{a\}\times\im{1-f}$ corresponds to the orbit $Qa$.
\end{itemize}
\end{example}

\begin{example}\label{Ex:4}
There are three quasi-affine quandles of order four, and all of them are affine:
\begin{align*}
\aff{\Z_4,1}\simeq\aff{\Z_2^2,1}&\simeq\ext{\Z_1,1,(0,0,0,0)},\\
\aff{\Z_4,-1}\simeq\aff{\Z_2^2,\left(\begin{smallmatrix} 0&1\\1&0\end{smallmatrix}\right)}&\simeq\ext{\Z_2,1,(0,1)},\\
\aff{\Z_2^2,\left(\begin{smallmatrix} 1&1\\1&0\end{smallmatrix}\right)}&\simeq\ext{\Z_2^2,\left(\begin{smallmatrix} 1&1\\1&0\end{smallmatrix}\right),(0)}.
\end{align*}
Consider a semiregular extension $\ext{A,f,(d_1,\dots,d_k)}$ such that $|A|\cdot k=4$.
If $A=\Z_1$, we have only one option, $\ext{\Z_1,1,(0,0,0,0)}$, which is a projection quandle.
If $A=\Z_2$, we have four options $\ext{\Z_2,1,(d_1,d_2)}$. The cases $(0,0)$ and $(1,1)$ result in the projection quandle, and it is easy to check that both cases $(0,1)$, $(1,0)$ are isomorphic to $\aff{\Z_4,-1}$.
If $A=\Z_2^2$ or $A=\Z_4$, the value of $\bar d=(d)$ is irrelevant, and we obtain one of the three affine quandles.
\end{example}

\begin{example}\label{Ex:6}
There are four quasi-affine quandles of order six (see Proposition \ref{p:enum_pq}):
\begin{align*}
\aff{\Z_6,1}\simeq\, &\ext{\Z_1,1,(0,0,0,0,0,0)},\\ 
&\ext{\Z_2,1,(0,0,1)},\\
&\ext{\Z_3,1,(0,1)},\\
\aff{\Z_6,-1}\simeq\, &\ext{\Z_3,2,(0,0)}.
\end{align*}
%\simeq\aff{\Z_3,-1}\times\mathrm{Proj}(2).\]
The second and third quandles are non-isomorphic, because they have different orbit sizes. They are not affine, since they violate condition (3) of Theorem \ref{thm:affine}. In the second case, $\im{1-f}=\{0\}$ and $\{0\}$ has two representatives in $\bar d$ whereas $\{1\}$ has only one. In the third case, $\im{1-f}=\{0\}$ and $\{2\}$ has no representative in~$\bar d$.
\end{example}

Free medial quandles can also be presented using semiregular extensions. This is essentially proved in \cite[Theorem 3.3]{JPZ}.

\begin{example}\label{ex:free}
Let $X$ be a set, $x_0\in X$ and denote $X^-=X\smallsetminus\{x_0\}$.
Let $A=\bigoplus_{x\in X^-}\Z[t,t^{-1}]$ be the free $\Z[t,t^{-1}]$-module over a free base $(e_x:x\in X^-)$ and put $e_{x_0}=0$.
Then the free medial quandle of rank $|X|$ can be represented as $\ext{A,t,\bar e}$, with the free base $((e_x,0):x\in X)$.
\end{example}

In Theorem \ref{thm:affine}(3), we represent affine quandles using a tuple $\bar d$ which is a \emph{multi}transversal. The following result explains the role of its multiplicity.

\begin{proposition}\label{p:prod}
Let $A$ be an abelian group, $f$ an automorphism of $A$, and $(d_i:i\in I)$ a multitransversal of $A/\im{1-f}$ of multiplicity $k$. 
Let $J\subseteq I$ such that $(d_i:i\in J)$ is a transversal of $A/\im{1-f}$. 
Then $\ext{A,f,(d_i:i\in I)}$ is isomorphic to the direct product \[\ext{A,f,(d_i:i\in J)}\times\mathrm{Proj}(k).\]
\end{proposition}

\begin{proof}
Let $K$ be a set of size $k$ and let $\xi$ be a bijection $J\times K\to I$ such that $d_i-d_{\xi(i,u)}\in\im{1-f}$, for every $i\in J$ and $u\in K$ (this is possible, because each block of $\im{1-f}$ contains the same number of representatives in $\bar d$). Choose witnesses $c_{i,u}\in A$ such that $d_i-d_{\xi(i,u)}=(1-f)(c_{i,u})$. Consider the mapping 
\begin{align*}
\Phi:\ext{A,f,(d_i:i\in J)}\times\mathrm{Proj}(k)&\to\ext{A,f,(d_i:i\in I)},\\ ((i,a),u)&\mapsto(\xi(i,u),\ a+c_{i,u}).
\end{align*}
This is clearly a bijection, and a straightforward calculation shows that it is an isomorphism:
\begin{align*}
\Phi((i,a),u)*&\Phi((j,b),v) 
= (\xi(i,u),\ a+c_{i,u})) * (\xi(j,v), b+c_{j,v}) \\
&= (\xi(j,v),\ (1-f)(a+c_{i,u})+f(b+c_{j,v})+d_{\xi(i,u)}-d_{\xi(j,v)}) \\
&= (\xi(j,v),\ (1-f)(a)+f(b)+ (1-f)(c_{i,u})+c_{j,v}-(1-f)(c_{j,v})+ d_{\xi(i,u)}-d_{\xi(j,v)}) \\
&= (\xi(j,v),\ (1-f)(a)+f(b)+ d_i-d_{\xi(i,u)}+c_{j,v}-d_j+d_{\xi(j,v)}+ d_{\xi(i,u)}-d_{\xi(j,v)}) \\
&= (\xi(j,v),\ (1-f)(a)+f(b)+ d_i-d_j+c_{j,v}) \\
&= \Phi((j,(1-f)(a)+f(b)+d_i-d_j),v)=\Phi(((i,a),u)*((j,b),v)).
\end{align*}
\end{proof}

As a consequence, we obtain an interesting decomposition theorem. It is related to \cite[Theorem 5.5]{JPSZ} which states a similar result covering all medial quandles: every medial quandle where all orbits are latin, is a direct product of an affine quandle and a projection quandle.

\begin{corollary}\label{cor:latin_orbits}
Let $A$ be an abelian group and $f$ its automorphism such that $1-f$ is onto. Then, for any tuple $\bar d$, the extension $\ext{A,f,\bar d}$ is isomorphic to 
the direct product of the affine quandle $\aff{A,f}$ and a projection quandle.
\end{corollary}

\begin{proof}
Since $\im{1-f}=A$, there is only one coset of $\im{1-f}$, hence $\bar{d}$ is a multitransversal of~$A/A$, and Proposition \ref{p:prod} applies. Clearly, $\ext{A,f,(d_1)}=\aff{A,f}$.
\end{proof}

\begin{remark}\label{rem:central}
The quandle $\ext{A,f,(d_i:i\in I)}$ is a \emph{central extension} of the projection quandle $(I,*)$ over the affine quandle $\aff{A,f}$, with the cocycle $\theta_{i,j}=d_i-d_j$. Here we mean central extensions in the sense of \cite[Chapter 7]{FM}. An ongoing project \cite{BS} aims at adapting the general theory of abelian and central extensions of \cite{FM} to quandles. Other definitions of abelian extensions of quandles exist in literature. For example, $\ext{A,f,(d_i:i\in I)}$ is a special kind of abelian extension in the sense of \cite{Ja}, but not in the sense of \cite{CENS,CSV}, where abelian extensions are defined as a special kind of coverings, restricting constant cocycles to abelian groups.
\end{remark}

\section{A universal algebraic characterization}\label{sec:ua}

In universal algebra \cite{Ber}, an algebraic structure $A$ is called \emph{abelian} if for every $(k+1)$-ary term operation $t$ and every $a,b,u_1,\dots,u_k,v_1,\dots,v_k\in A$, the following implication holds:
\[ t(a,u_1,\dots,u_k)=t(a,v_1,\dots,v_k)\ \Rightarrow\ t(b,u_1,\dots,u_k)=t(b,v_1,\dots,v_k)\]
The condition may look ad hoc, but it has a natural meaning: It is not hard to prove that an algebra $A$ is abelian if and only if, in the direct power $A^2$, the diagonal $\{(a,a):a\in A\}$ is a block of a congruence \cite[Theorem 7.30]{Ber}. Consequently, a group $G$ is abelian in the present sense (i.e., the diagonal is a normal subgroup of $G^2$) if and only if $G$ is abelian in the usual sense. A ring is abelian if and only if it is a zero ring ($xy=0$ for every $x,y$).

It turns out that this syntactic condition is closely related to representability of general algebraic structures by modules (see \cite{SS,Sz} for a detailed account).
In one direction, consider a module $M$ over a ring $R$. Every term operation $t(x,x_1,\dots,x_k)$ can be written as $rx+\sum_{i=1}^k r_ix_i$ for some $r,r_i\in R$.
Then $t(a,u_1,\dots,u_k)=t(a,v_1,\dots,v_k)$ implies that $\sum_{i=1}^k r_iu_i=\sum_{i=1}^k r_iv_i$, and thus we have
$t(b,u_1,\dots,u_k)=t(b,v_1,\dots,v_k)$ for every $b\in M$. Hence every module is abelian. The same argument shows that every algebraic structure defined by term operations over a module is abelian, too. And indeed, a subalgebra of an abelian algebra is also abelian. We just proved that every quasi-affine algebraic structure is abelian.

The converse implication is more complicated. It holds in many particular cases \cite{SS,Sz}, but not in general \cite{Qu}. As the first step towards establishing that abelian quandles are quasi-affine, we prove that they satisfy condition (2) of Theorem \ref{thm:quasi-affine}.

\begin{lemma}\label{l:semireg-abelian}
Let $Q$ be an abelian quandle. Then $\dis Q$ is a semiregular abelian group.
\end{lemma}

\begin{proof}
First we show that $Q$ is medial, and thus $\dis Q$ is abelian. Let $t(x,y,u,v)=(x*y)*(u*v)$ and consider any $a,b,c,d\in Q$. Using idempotence and left distributivity,
\[t(a,a,b,c)=(a*a)*(b*c)=(a*b)*(a*c)=t(a,b,a,c).\]
Using abelianness, we obtain
\[(d*a)*(b*c)=t(d,a,b,c)=t(d,b,a,c)=(d*b)*(a*c).\]

For semiregularity, consider $\alpha=L_{a_1}L^{-1}_{b_1}\ldots L_{a_n}L^{-1}_{b_n}\in\Dis(Q)$ such that $\alpha(c)=c$ for some $c\in Q$. We shall prove that $\alpha$ is the identity mapping. Let $t(z,x_1,\dots,x_n,y_1,\dots,y_n)$ be the term that represents the formal expression $L_{x_1}L^{-1}_{y_1}\ldots L_{x_n}L^{-1}_{y_n}(z)$.
Then
\[t(c,a_1,\dots,a_n,b_1,\dots,b_n)=\alpha(c)=c=L_{c}L^{-1}_{c}\ldots L_{c}L^{-1}_{c}(c)=t(c,c,\dots,c,c,\dots,c).\]
Using abelianness, we obtain
\[\alpha(d)=t(d,a_1,\dots,a_n,b_1,\dots,b_n)=t(d,c,\dots,c,c,\dots,c)=L_{c}L^{-1}_{c}\ldots L_{c}L^{-1}_{c}(d)=d\]
for every $d\in Q$.
\end{proof}

\begin{remark}
We just proved that abelian quandles are medial. In the next section, we will prove that abelian quandles are quasi-affine. 
A different but related affine representation result was established by Kearnes in \cite[Theorem 1.5]{K}: 
\emph{Every medial quandle $Q$ admits a strongly abelian congruence $\theta$ such that $Q/\theta$ is quasi-affine.}
It is not difficult to prove that, in quandles, strongly abelian congruences are precisely those below the kernel of the Cayley mapping $x\mapsto L_x$
(see \cite{BS}).
\end{remark}

\section{Characterization theorems}\label{sec:proofs}

In the present section, we prove the main results. We start with the characterization theorem for affine quandles.

\begin{proof}[Proof of Theorem \ref{thm:affine}]

$(1)\Rightarrow(2)$.
Let $Q=\aff{A,f}$. Recall that 
\[\dis Q=\{x\mapsto a+x:\ a\in\im{1-f}\}.\]
It immediately follows that $\dis Q$ is abelian and semiregular.
Next we prove that 
\[[\dis Q,L_e]=\{x\mapsto b+x:\ b\in\im{(1-f)^2}\}.\]
For $\alpha\in\dis Q$, $\alpha(x)=a+x$, we have
\[ [\alpha,L_e](x)=a+(1-f)(e)+f(-a+(1-f^{-1})(e)+f^{-1}(x))=(1-f)(a)+x.\]
Indeed, $(1-f)(a)\in\im{(1-f)^2}$, and every element of $\im{(1-f)^2}$ can be written this way.

Now we finish the proof of (2). Fix $e\in\dis Q$ and a transversal $T$ of the orbit decomposition, that is, a transversal of $A/\im{1-f}$. We shall prove that $\{\{L_aL_e^{-1}:a\in T\}\}$ is a multitransversal of $\dis Q/[\dis Q,L_e]$. 
Since $L_aL_e^{-1}(x)=(1-f)(a-e)+x$, this is equivalent to the fact that
$\{\{(1-f)(a-e):a\in T\}\}$ is a multitransversal of $\im{1-f}/\im{(1-f)^2}$. Now apply Lemma \ref{l:mtr2} to $G=A$, $\varphi=1-f$ and the transversal $\{a-e:a\in T\}$ of $A/\im{1-f}$.

$(2)\Rightarrow(2')$ is trivial.

$(2')\Rightarrow(3)$.
Let $A=\dis Q$, $f(\alpha)=\alpha^{L_e}$ and put $d_x=L_xL_e^{-1}$, $x\in T$. Observe that $[\dis Q,L_e]=\im{1-f}$, because
\[ [\alpha,L_e]=\alpha L_e\alpha^{-1} L_e^{-1}=\alpha f(\alpha)^{-1}=(1-f)(\alpha).\]
According to Lemma \ref{l:ext2}, $Q$ is isomorphic to the extension $\ext{\dis Q,f,(L_xL_e^{-1}:x\in T)}$. By assumptions, $\bar d$ is a multitransversal of $\dis Q/[\dis Q,L_e]=A/\im{1-f}$.

$(3)\Rightarrow(1)$.
According to Proposition~\ref{p:prod}, $Q$ is a product of a projection quandle (which is affine) and an extension $\ext{A,f,\bar e}$ where~$\bar e$ is a transversal. Since the product of affine quandles is affine, it remains to prove that the implication holds assuming that $\bar d$ is a transversal.

According to the Hou-\v S\v tov\'\i\v cek Lemma \ref{l:hs} for $\varphi=1-f$, there is an abelian group $E\geq A$ and an epimorphism $\psi:E\to A$ such that $\psi|_A=1-f$ and $\psi/A:E/A\simeq A/\im{1-f}$. Let $g=1-\psi$. First we prove that $g\in\aut E$. Indeed, it is an endomorphism.
Given $y\in E$, we will find all $x\in E$ such that $g(x)=(1-\psi)(x)=y$, that is, such that $\psi(x)=x-y$. Since $\im\psi=A$, we must have $x-y\in A$, and thus we can assume that $x=y+a$ for some $a\in A$. Now, on one hand, we have
\[\psi(y+a)=\psi(x)=x-y=y+a-y=a,\]
and on the other hand, we have
\[\psi(y+a)=\psi(y)+\psi(a)=\psi(y)+a-f(a),\]
because $\psi|_A=1-f$. Putting together, $x=y+a$ is a solution to the equation $g(x)=y$ if and only if $\psi(y)=f(a)$, that is, if and only if $a=f^{-1}\psi(y)$. Therefore, the equation has a unique solution, and thus $g$ is bijective.

Consider a transversal $(e_i)_{i\in I}$ of $E/A$ such that $\psi(e_i)=d_{i}$. %(recall that $d_{B,i}=d_{B,j}\in B$ for every $i,j$). 
Define a mapping
\[\Phi:I\times A\to  E,\qquad (i,a)\mapsto e_i+a.\]
We prove that this is a quandle isomorphism $\ext{A,f,\bar d}\simeq\aff{E,g}$.
It is indeed bijective: given $u\in E$, there is a unique decomposition $u=e_i+a$ where $e_i$ is the representative of the coset such that $u\in e_i+A$, and thus $(i,a)$ is the unique preimage. To show that $\Phi$ is a homomorphism, we calculate
\begin{align*}
\Phi(i,a)*\Phi(j,b)&=(e_i+a)*(e_j+b)\\
&=(1-g)(e_i+a)+g(e_j+b)\\
&=(1-g)(a)+g(b)+(1-g)(e_i)-(1-g)(e_j)+e_j\\
&=(1-f)(a)+f(b)+\psi(e_i)-\psi(e_j)+e_j\\
&=(1-f)(a)+f(b)+d_i-d_j+e_j\\
&=\Phi(j,(1-f)(a)+f(b)+d_i-d_j)\\
&=\Phi((i,a)*(j,b))
\end{align*}
for every $i,j\in I$ and $a,b\in A$.
\end{proof}

Now, using Theorem \ref{thm:affine}, we can prove the characterization theorem for quasi-affine quandles.

\begin{proof}[Proof of Theorem \ref{thm:quasi-affine}]

$(1)\Rightarrow(2)$.
Let $Q$ be a subquandle of $\aff{A,f}=\ext{A,f,(0)}$. According to Lemma \ref{l:ext1}, $G=\dis{\aff{A,f}}$ is an abelian semiregular permutation group. The group $\dis Q$ is a subgroup of permutations from $G$ restricted to the subset $Q\subseteq A$, hence it is also abelian and semiregular.

$(2)\Rightarrow(3)$ was proved in Lemma \ref{l:ext2}.

$(3)\Rightarrow(1)$.
Assume that $Q=\ext{A,f,\bar d}$ where $\bar d=(d_i:i\in I)$. Extend the set $I$ and the tuple $\bar d$ into a set $J\supseteq I$ and a tuple $\bar e=(e_j:j\in J)$ such that $\bar e$ is a multitransversal of $A/\im{1-f}$ and $e_i=d_i$ for every $i\in I$. 
Then $Q=\ext{A,f,\bar d}$ is a subquandle of $\ext{A,f,\bar e}$, which is affine according to Theorem~\ref{thm:affine}.

(Extending the tuple $\bar d$ is indeed possible: from each coset $x+\im{1-f}$, add sufficiently many elements so that all cosets have the same number of representatives. Here is a formal description.
Choose a transversal $T$ of $A/\im{1-f}$. For $x\in T$, let $n_x=|\bar d\cap(x+\im{1-f})|$ and put $n=\sup\{n_x:x\in T\}$.
Let $J=I\cup\bigcup_{x\in T} J_x$, where $J_x$ are pairwise disjoint sets, disjoint with $I$, such that $|J_x|+n_x=n$.
Define $e_i=d_i$ for every $i\in I$, and for every $j\in J_x$, choose $e_j\in x+\im{1-f}$ arbitrarily.)

$(1)\Rightarrow(4)$ holds for all algebraic structures, see Section \ref{sec:ua}.

$(4)\Rightarrow(2)$ was proved in Lemma \ref{l:semireg-abelian}.
\end{proof}

\begin{corollary}
Let $Q$ be a quasi-affine quandle. Then there exists an abelian group $A$ and its automorphism $f$ such that $Q$ embeds into $\aff{A,f}$ and $|A|\leq|Q|\cdot|\dis Q|\leq|Q|^2$.
\end{corollary}

\begin{proof}
According to Lemma \ref{l:ext2}, $Q\simeq\ext{D,g,\bar d}$ such that $D=\dis Q$ and $\bar d$ is indexed by $T$, a set of orbit representatives. In particular, $|Q|=|\dis Q\times T|=|\dis Q|\cdot|T|$. Using the construction from the proof of Theorem \ref{thm:quasi-affine}, $(3)\Rightarrow(1)$, $Q$ embeds into an affine quandle $R=\ext{D,g,\bar e}$ where $\bar e$ is indexed by a set not larger than $|T|\cdot |D/\im{1-f}|\leq |T|\cdot|\dis Q|=|Q|$. Therefore, $|R|\leq|\dis Q|\cdot|Q|\leq|Q|^2$.
\end{proof}

For \emph{finite} quandles, the balancedness conditions of Theorem \ref{thm:affine} can be formulated alternatively, perhaps more esthetically. In any affine quandle, 
every column of the multiplication table contains the same number of occurrences of each entry. For finite quasi-affine quandles, the converse holds, too.

Let $m_{x,y}$ denote the number of occurrences of $y$ in the column of $x$ in the multiplication table of~$Q$; formally,
\[m_{x,y}=|\{z\in Q:\ z*x=y\}|.\]
Indeed $m_{x,y}=0$ if $y\not\in Qx$.

\begin{prop}\label{p:occurrences}
If $Q$ is an affine quandle, then $m_{x,y_1}=m_{x,y_2}$ for every $x\in Q$ and every $y_1,y_2\in Qx$.
\end{prop}

\begin{proof}
Let $Q=\aff{A,f}$. Then $m_{x,y}=|\{z\in Q:\ (1-f)(z)=y-f(x)\}|$. Assuming that $y\in Qx$, there is $u\in A$ such that $y=(1-f)(u)+x$, and thus 
\[m_{x,y}=|\{z\in Q:\ (1-f)(z)=(1-f)(x+u)\}|=|\Ker{1-f}|,\] 
because $(1-f)(a)=(1-f)(a')$ if and only if $a-a'\in\Ker{1-f}$. We see that $m_{x,y}$ is independent of $y$.
\end{proof}

\begin{theorem}\label{thm:affine_fin}
The following statements are equivalent for a finite quasi-affine quandle $Q$: 
\begin{enumerate}
\item $Q$ is affine;
\item for every $x\in Q$ and every $y_1,y_2\in Qx$, $m_{x,y_1}=m_{x,y_2}$;
\item there exists $x\in Q$ such that for every $y_1,y_2\in Qx$, $m_{x,y_1}=m_{x,y_2}$.
\end{enumerate}
\end{theorem}

\begin{proof}
$(1)\Rightarrow(2)$ is Proposition \ref{p:occurrences} and $(2)\Rightarrow(3)$ is trivial.

$(3)\Rightarrow(1)$.
Assume that $Q=\ext{A,f,\bar{d}}$ is an indecomposable extension and take $j\in I$, $b\in A$ such that $x=(j,b)$. Consider $y=(j,c)\in Qx$. We have
\[m_{(j,b),(j,c)}=|\{(i,a):\ (i,a)*(j,b)=(j,c)\}|=|\{(i,a):\ (1-f)(a)+f(b)+d_i-d_j=c\}|.\]
Given $i,j,b,c$, the number of $a$'s satisfying the equation is precisely $|\Ker{1-f}|$, hence
\[m_{(j,b),(j,c)}=|\Ker{1-f}|\cdot|\{i:\ d_i-d_j+f(b)-c\in\im{1-f}\}|.\]
Denoting $u_c=d_j-f(b)+c$, we obtain
\[m_{(j,b),(j,c)}=|\Ker{1-f}|\cdot|\{i:\ d_i\in u_c+\im{1-f}\}|.\]
According to (3), this expression shall be independent of $y=(j,c)$. Running over all $c\in A$, the element $u_c$ also runs over all elements of $A$, and thus all cosets of $\im{1-f}$ must contain the same number of $d_i$'s. In other terms, $\bar d$ is a multitransversal of $A/\im{1-f}$.
According to Theorem~\ref{thm:affine}, the quandle~$Q$ is affine.
\end{proof}

We used finiteness only in the last step of the proof (independence of $m_{(j,b),(j,c)}$ on $(j,c)$ implies independence of the right factor on $(j,c)$), and it would have been sufficient to assume only that~$\mathrm{Ker}(1-f)$ is finite. Therefore, we could have only assumed that $m_{x,x}$ is finite for some $x$. Without this assumption, the implication fails, as witnessed by the following example.

\begin{example}
Let $Q=\ext{\Z,1,\bar{d}}$, where $\bar{d}$ contains every non-zero integer once and zero twice. 
Then $\bar d$ is not a multitransversal of $\Z/\im{1-f}=\Z/0$,
and this implies, as we shall see in Proposition~\ref{thm:affbalanced}, that $Q$ is not affine.
But $m_{x,y}$ is infinite countable for every $x,y$, hence $Q$ satisfies condition~(2) of Theorem~\ref{thm:affine_fin}.
\end{example}

The following example shows that the implication $(2)\Rightarrow(1)$ of Theorem~\ref{thm:affine_fin} does not hold without the assumption that $Q$ is quasi-affine. 

\begin{example}
Let $Q$ be the quandle defined by the following multiplication table:
$$\begin{array}{c|cccccccc}
Q & 1 & 2 & 3 & 4 & 5 & 6 & 7 & 8\\
\hline
1 & 1 & 2 & 3 & 4 & 5 & 6 & 7 & 8\\
2 & 1 & 2 & 3 & 4 & 5 & 6 & 7 & 8\\
3 & 2 & 1 & 3 & 4 & 6 & 5 & 8 & 7\\
4 & 2 & 1 & 3 & 4 & 6 & 5 & 8 & 7\\
5 & 2 & 1 & 4 & 3 & 5 & 6 & 8 & 7\\
6 & 2 & 1 & 4 & 3 & 5 & 6 & 8 & 7\\
7 & 1 & 2 & 4 & 3 & 6 & 5 & 7 & 8\\
8 & 1 & 2 & 4 & 3 & 6 & 5 & 7 & 8
\end{array}.$$
It satisfies condition (2) of Theorem~\ref{thm:affine_fin} but it is not quasi-affine, since $\dis Q$ is not semiregular. 
%Indeed, $\dis{Q}=\left<(21)(65)(87),(43)(65)\right>\cong \Z_2^2$ and~$Q$ is not quasi-affine.
(The example is a 2-reductive medial quandle and it was constructed using the methods of \cite[Section 6]{JPSZ}.)
\end{example}

The following example shows that it is not possible to characterize affine quandles by a first-order property of the displacement group. 

\begin{example}
Let $Q_1=\ext{\Z_3,1,(0,1)}$ and $Q_2=\ext{\Z_3,2,(0,0)}=\aff{\Z_6,-1}$. Then both $\dis{Q_1}$ and $\dis{Q_2}$ are isomorphic to $\Z_3$, but $Q_1$ is not affine (see Example \ref{Ex:6}).
\end{example}

Affine quandles have tiny displacement groups, but the converse is not true. In fact, the conditions ``$\dis Q$ is tiny'' and ``$Q$ is quasi-affine'' are independent for medial quandles, as witnessed by the following example.

\begin{example}\ 
\begin{itemize}
	\item The quandle $Q=\ext{\Z_2,1,(0,0,1)}$ is quasi-affine, $\dis Q$ is tiny but~$Q$ is not affine.
	\item The quandle $Q=\ext{\Z_3,1,(0,1)}$ is quasi-affine but $\dis Q$ is not tiny.
	\item The quandle $Q=\aff{\Z_4,-1}/\alpha$, where $\alpha$ is the congruence with blocks $\{0,2\}$, $\{1\}$, $\{3\}$, is not quasi-affine, because $\dis Q$ does not act semiregularly (it fixes the block $\{0,2\}$), but $\dis Q$ is tiny.
\end{itemize}
\end{example}

In the end, we also find worth mentioning the following characterization which follows from the preceding theory.

\begin{corollary}\label{r:homim}
A quandle is a homomorphic image of a quasi-affine quandle if and and only if it is medial.
\end{corollary}

\begin{proof}
Mediality is clearly a necessary condition, and the other direction follows immediately from Theorem \ref{thm:quasi-affine} and Lemma \ref{l:ext2}: every medial quandle is a homomorphic image of some $E=\ext{A,f,\bar d}$.
\end{proof}

\section{Algorithms}\label{sec:alg}

\def\krok#1#2{#1 & #2 \\}
\def\io#1#2{{$\ $}\\ \begin{tabular}{ll} {\bf In:} & #1 \\ {\bf Out:} & #2 \end{tabular} \\}
\def\afont#1{{\bf #1}}

We will discuss two decision problems. On input, we have a quandle. We are asked to decide whether the quandle is affine or quasi-affine, respectively.
We will assume that the input quandle is in the form of a multiplication table, although one can imagine other representations for which the algorithms work efficiently.
Both algorithms are based on the properties of the displacement group, as described in conditions (2) of Theorems \ref{thm:quasi-affine} and \ref{thm:affine}. Since the input is finite, we will check balancedness using Theorem \ref{thm:affine_fin}. 
Let us start with the affine case.

\begin{algorithm}\label{alg:affine}
\noindent
\io{a quandle $Q$}{is $Q$ affine?}
\begin{tabular}{ll}
\krok{1.}{pick $e\in Q$}
\krok{2.}{$D:=\{L_xL_e^{-1}: x\in Q\}$}
\krok{3.}{\afont{for each} $\alpha\in D$ \afont{do}}
\krok{4.}{\quad \afont{if} $0<|Fix(\alpha)|<|Q|$ \afont{then return} false}
\krok{5.}{\quad \afont{for each} $\beta\in D$ \afont{do}}
\krok{6.}{\qquad \afont{if} $\alpha\beta\neq\beta\alpha$ \afont{then return} false}
\krok{7.}{\qquad \afont{if} $\alpha\beta\not\in D$ \afont{then return} false}
\krok{8.}{$m_x:=0$ for each $x\in Q$}
\krok{9.}{\afont{for each} $x\in Q$ \afont{do} $m_{xe}:=m_{xe}+1$}
\krok{10.}{\afont{for each} $x\in Q$ \afont{do} \afont{if} $m_{xe}\neq m_{e}$ \afont{then return} false}
\krok{11.}{\afont{return} true}
\end{tabular}
\end{algorithm}

In the first part of the algorithm (lines 1--7), it is checked whether $\dis Q$ is semiregular, abelian and tiny. All of these are necessary conditions for a quandle to be affine, and sufficient to be quasi-affine. If succeeded, the algorithm checks condition (2) of Theorem \ref{thm:affine_fin}, picking the column $e$. We use an observation that if $\dis Q$ is tiny then $Qe=\{xe:x\in Q\}$.

\begin{proposition}\label{p:alg-affine}
Algorithm \ref{alg:affine} runs in $\bigo(n^3\log n)$ time with respect to $n=|Q|$. 
\end{proposition}

\begin{proof}
All operations performed with permutations on $Q$ (comparison, composition, counting fixed points) can be calculated in $\bigo(n\log n)$ time. 
In the first part (lines 1--7), we run over $n^2$ pairs of permutations $\alpha,\beta$, performing a fixed amount of operations with them, resulting in $\bigo(n^3\log n)$ time. The remaining part of the algorithm takes essentially linear time.
\end{proof}

In the quasi-affine case, we do not have the convenient condition that the displacement group is tiny.
To avoid a blow-up during calculation of $\dis Q$ when the input is not quasi-affine, we implement a convenient upper bound on $|\dis Q|$ under the assumption of semiregularity.

\begin{lemma}\label{l:alg}
Let $Q$ be a quandle with $\dis Q$ acting semiregularly. Then $|\dis Q|\leq|Q|$.
\end{lemma}

\begin{proof}
For any group, we have $|G|=|G_e|\cdot|Orb(e)|$. In particular, $|\dis{Q}|=|\dis{Q}_e|\cdot|Qe|$. If $\dis Q$ acts semiregularly, then $|\dis Q_e|=1$, and thus $|\dis Q|=|Qe|\leq|Q|$.
\end{proof}

\begin{algorithm}\label{alg:quasi-affine}
\noindent
\io{a quandle $Q$}{is $Q$ quasi-affine?}
\begin{tabular}{ll}
\krok{1.}{pick $e\in Q$ }
\krok{2.}{$D:=\{L_xL_e^{-1}:x\in Q\}$}
\krok{3.}{\afont{for each} $\alpha\in D$ \afont{do}}
\krok{4.}{\quad \afont{if} $0<|Fix(\alpha)|<|Q|$ \afont{then return} false}
\krok{5.}{\quad \afont{for each} $\beta\in D$ \afont{do}}
\krok{6.}{\qquad \afont{if} $\alpha\beta\neq\beta\alpha$ \afont{then return} false}
\krok{7.}{$P:=\{\{\alpha,\beta\}:\alpha,\beta\in D\}$}
\krok{8.}{\afont{while} $P\neq\emptyset$ \afont{do} }
\krok{9.}{\quad select $\{\alpha,\beta\}\in P$, remove $\{\alpha,\beta\}$ from $P$}
\krok{10.}{\quad \afont{if} $\alpha\beta\not\in D$ \afont{then}}
\krok{11.}{\qquad \afont{if} $|D|\geq |Q|$ \afont{then return} false}
\krok{12.}{\qquad \afont{if} $0<|Fix(\alpha\beta)|<|Q|$ \afont{then return} false}
\krok{13.}{\qquad $D:=D\cup\{\alpha\beta\}$}
\krok{14.}{\qquad $P:=P\cup\{\{\alpha\beta,\delta\}:\delta\in D\}$}
\krok{15.}{\afont{return} true}
\end{tabular}
\end{algorithm}

In the first part of the algorithm (lines 1--6), we consider the generators of $\dis Q$ and check whether they commute and have the correct number of fixed points. If succeeded, on lines 7--14, the algorithm generates $\dis Q$ in a standard way.
(In the expression $\{\alpha,\beta\}$, we allow $\alpha=\beta$. Then the result is a one-element set that represents the mapping $\alpha^2$.)
Whenever we find a composition $\alpha\beta$ not yet in $D$, we check whether it has the correct number of fixed points, and if so, $\alpha\beta$ is added to $D$ and $P$ is expanded accordingly (no need to check commutativity, since a group is abelian if and only if its generators commute). The algorithm terminates on line 11 if it realizes that $|\dis Q|>|Q|$, thanks to Lemma \ref{l:alg}.

\begin{proposition}\label{p:alg-quasi-affine}
Algorithm \ref{alg:quasi-affine} runs in $\bigo(n^3\log n)$ time with respect to $n=|Q|$. 
\end{proposition}

\begin{proof}
As in Proposition \ref{p:alg-affine}, the first part (lines 1--6) results in $\bigo(n^3\log n)$ time. 
In the second part, we start with $|P|\leq n^2$. The condition on line 10 is satisfied at most $n$ times, thanks to Lemma \ref{l:alg} employed on line 11. Therefore, during the run of the algorithm, at most $n^2$ unordered pairs are added to $|P|$ on line 14, an the loop will finish after at most $2n^2$ steps. Each step only does a fixed amount of operations with permutations on $Q$, hence the loop requires $\bigo(n^3\log n)$ time.
\end{proof}

%Both algorithms run in essentially cubic time with respect to $n=|Q|$. That is, they are subquadratic with respect to the input size (which is $\bigo(n^2\log n)$).

Both algorithms solve the decision problems (the answer is yes/no). To output the actual affine representation is a more complex problem. We can indeed retrieve a representation in the form of a semiregular extension, as suggested by Lemma \ref{l:ext2}, as all information is readily available. Given $Q=\ext{A,f,\bar d}$, the proof of Theorem \ref{thm:quasi-affine} explains how to expand $\bar d$ so that the result is an affine quandle into which $Q$ embeds; this expansion can be implemented efficiently. However, given an affine quandle in the form of a semiregular extension, it is not clear how to obtain the actual affine representation $\Aff(E,g)$. In the proof of Theorem \ref{thm:affine}, we used the Hou-\v S\v tov\'\i\v cek lemma to find the pair $(E,g)$, but its proof is not constructive and cannot be transformed into an algorithm.

\begin{remark}
Universal algebra suggests an alternative approach to checking whether a given quandle is quasi-affine. Theorem \ref{thm:quasi-affine} states that a quandle $Q$ is quasi-affine if and only if it is abelian (in the universal algebraic sense). Then $Q$ is abelian if and only if the diagonal is a block of a congruence in the direct power $Q^2$, that is, if and only if the congruence of $Q^2$ generated by the diagonal has the diagonal as its block. Congruences of binary algebras can be generated in cubic time with respect to the number of elements, see \cite{DDK} for a fast algorithm. This provides an alternative to Algorithm \ref{alg:quasi-affine}, but the time complexity seems to be much worse.
\end{remark}

\section{Affine meshes and the isomorphism theorem}\label{sec:iso}

In \cite{JPSZ}, we developed a representation of medial quandles by certain heterogeneous affine structure, called affine meshes. 
First, we recall essential constructions and results.

\begin{de}
An \emph{affine mesh} over a non-empty set $I$ is a triple
$$\mathcal A=((A_i)_{i\in I},\,(\varphi_{i,j})_{i,j\in I},\,(c_{i,j})_{i,j\in I})$$ where $A_i$ are abelian groups, $\varphi_{i,j}:A_i\to A_j$ homomorphisms, and $c_{i,j}\in A_j$ constants, satisfying the following conditions for every $i,j,j',k\in I$:
\begin{enumerate}
    \item[(M1)] $1-\varphi_{i,i}$ is an automorphism of $A_i$;
    \item[(M2)] $c_{i,i}=0$;
    \item[(M3)] $\varphi_{j,k}\varphi_{i,j}=\varphi_{j',k}\varphi_{i,j'}$, i.e., the following diagram commutes:
$$\begin{CD}
A_i @>\varphi_{i,j}>> A_j\\ @VV\varphi_{i,j'}V @VV\varphi_{j,k}V\\
A_{j'} @>\varphi_{j',k}>> A_k
\end{CD}$$
    \item[(M4)] $\varphi_{j,k}(c_{i,j})=\varphi_{k,k}(c_{i,k}-c_{j,k})$.
\end{enumerate}
The mesh is called \emph{indecomposable} if
%\[A_j=\left<\bigcup_{i\in I}\left(c_{i,j}+\im{\varphi_{i,j}}\right)\right>,\] for every $j\in I$.
%Equivalently, the
for every $j\in I$, the group $A_j$ is generated by the set \[\{c_{i,j},\varphi_{i,j}(a):\ i\in I,a\in A_i\}.\]
If the index set is clear from the context, we shall write briefly $\mathcal A=(A_i,\varphi_{i,j},c_{i,j})$.
\end{de}

\begin{de}
The \emph{sum of an affine mesh} $(A_i,\varphi_{i,j},c_{i,j})$ over a set $I$ is an algebraic structure defined on the disjoint union of the sets $A_i$ by
%\begin{align*}
\[a\ast b=c_{i,j}+\varphi_{i,j}(a)+(1-\varphi_{j,j})(b),\]
%a\ld b&=(1-\varphi_{j,j})^{-1}(b-\varphi_{i,j}(a)-c_{i,j}).
%\end{align*}
for every $a\in A_i$ and $b\in A_j$.
\end{de}

The sum $Q$ of an affine mesh is a medial quandle. The fibers $A_i$, $i\in I$, form subquandles of $Q$ which are affine, namely, $\aff{A_i,1-\varphi_{i,i}}$. If the mesh is indecomposable, then the fibers are precisely the orbits of $Q$.

\begin{theorem}\cite[Theorem 3.14]{JPSZ}\label{thm:decomposition}
An algebraic structure $(Q,*)$ is a medial quandle if and only if it is the sum of an indecomposable affine mesh.
\end{theorem}

A quasi-affine quandle is medial, hence it is the sum of an affine mesh. The structure of the mesh is easily derived from the parameters of the semiregular extension.

\begin{observation}\label{o:ext}
A semiregular extension $\ext{A,f,(d_i:i\in I)}$ is the sum of an affine mesh \[((A)_{i\in I},\,(1-f)_{i,j\in I},\,(d_i-d_j)_{i,j\in I}).\]
The mesh (and thus the extension) is indecomposable if and only if the group $A$ is generated by the set 
\[\im{1-f}\cup\{d_i-d_j:\ i,j\in I\}.\]
\end{observation}

\begin{de}
Two affine meshes, $\mathcal A=(A_i,\varphi_{i,j},c_{i,j})$ over an index set $I$,
and $\mathcal A'=(A_i',\varphi_{i,j}',c_{i,j}')$ over an index set $I'$, are called \emph{homologous}, if there exist a bijection $\pi:I\to I'$, group isomorphisms $\psi_i:A_i\to A_{\pi (i)}'$, and constants $e_i\in A_{\pi (i)}'$, such that, for every $i,j\in I$,
\begin{enumerate}
    \item[(H1)] $\psi_j\varphi_{i,j}=\varphi_{\pi (i),\pi (j)}'\psi_i$, i.e., the following diagram commutes:
$$ \begin{CD}
A_i @>\varphi_{i,j}>> A_j\\ @VV\psi_iV @VV\psi_jV\\
A_{\pi (i)}' @>\varphi_{\pi (i),\pi (j)}'>> A_{\pi (j)}'
\end{CD}$$
    \item[(H2)] $\psi_j(c_{i,j})=c_{\pi (i),\pi (j)}'+\varphi_{\pi (i),\pi (j)}'(e_i)-\varphi_{\pi (j),\pi (j)}'(e_j)$.
\end{enumerate}
\end{de}

\begin{theorem}\cite[Theorem 4.2]{JPSZ} \label{thm:isomorphism}
Let $\mathcal A=(A_i,\varphi_{i,j},c_{i,j})$ and $\mathcal A'=(A_i',\varphi_{i,j}',c_{i,j}')$ be two indecomposable affine meshes.
The sums of $\A$ and $\A'$ are isomorphic quandles if and only if the meshes $\A$, $\A'$ are homologous.
\end{theorem}

Specializing Theorem \ref{thm:isomorphism}, we obtain an isomorphism theorem for indecomposable semiregular extensions.

\begin{theorem}\label{th:Isom_Ab}\label{thm:iso}
Let $\ext{A,f,(d_i:i\in I)}$ and $\ext{A',f',(d_i':i\in I')}$ be two indecomposable extensions. They are isomorphic if and only if
there exist a bijection $\pi:I\to I'$, an isomorphism $\psi:A\to A'$, and an element $a\in A'$ such that
\begin{itemize}
	\item[(E1)] $\psi f=f'\psi$,
	\item[(E2)] $\psi(d_i)-d_{\pi(i)}'\in a+\im{1-f'}$ for every $i\in I$.
%	\item[(E2)] $\psi(d_{i,1})=d_{\pi(i),2}+a+b_{\pi(i)}$ for every $i\in I_1$.
 \end{itemize}
\end{theorem}

\begin{proof}
Combining Observation \ref{o:ext} and Theorem \ref{thm:isomorphism}, we see that the two extensions are isomorphic if and only if there exist 
a bijection $\pi:I\to I'$, isomorphisms $\psi_i:A\to A'$, and elements $e_i\in A'$ such that for every $i,j\in I$
\begin{itemize}
	\item[(H1)] $\psi_i(1-f)=(1-f')\psi_j$,
	\item[(H2)] $\psi_j(d_i-d_j)=d_{\pi(i)}'-d_{\pi(j)}'+(1-f')(e_i)-(1-f')(e_j)$.
 \end{itemize}

$(\Rightarrow)$. Consider $\pi,\psi_i$ and $e_i$ as above. Choose $k\in I$ and define $\psi=\psi_k$ and $a=\psi_k(d_k)-d_{\pi(k)}'-(1-f')(e_k)$. Then condition (E1) is a special case of (H1) for $i=j=k$, and condition (E2) follows from (H2) with $j=k$, since 
$\psi(d_i)=\psi(d_k)+d_{\pi(i)}'-d_{\pi(k)}'+(1-f')(e_i)-(1-f')(e_k)=d_{\pi(i)}'+(1-f')(e_i)+a$, and thus
$\psi(d_i)-d_{\pi(i)}'\in a+\im{1-f'}$.

$(\Leftarrow)$. Consider $\pi,\psi,a$ as in the statement of the theorem. For every $i\in I$ define $\psi_i=\psi$ and select $e_i$ such that $\psi(d_i)-d_{\pi(i)}'=a+(1-f')(e_i)$.
Then condition (H1) immediately follows from  (E1) and the fact that $\psi_i=\psi_j$ for every $i,j$, and condition (H2) immediately follows from (E2) applied to both $i$ and $j$ (the two occurrences of $a$ in the expression cancel).
\end{proof}

In Theorem~\ref{thm:affine} we claim that affine quandles have at least one balanced
extension. Actually, all its indecomposable extensions are balanced. (The claim
is not true for decomposable extensions since, e.g., $\ext{\Z_3,1,\{0\}}$ is affine but not balanced.)

\begin{proposition}\label{thm:affbalanced}
Let $Q$ be affine and isomorphic to an indecomposable extension $\ext{A,f,\bar d}$, for some $A,f,\bar d$.
Then $\bar d$ is balanced.
\end{proposition}
\begin{proof}
By Theorem~\ref{thm:affine}, there exists an indecomposable extension $\ext{A',f',(d_i':i\in I')}$ isomorphic to $Q$
with $\bar d'$ balanced.
By Theorem \ref{thm:iso}, there exist a bijection $\pi:I'\to I$, a group isomorphism $\psi:A'\to A$, and an element $a\in A$ such that
conditions (E1) and (E2) hold. Since $\bar d'=(d_i':i\in I)$ is a balanced tuple of elements $A$, then $(\psi(d'_i):i\in I)$ is balanced,
i.e., it is a multitransversal of $A/\im{1-f}$. This implies that $(\psi(d_i')-a:i\in I')$ is also a multitransversal of $A/\im{1-f}$.
Now $\psi(d_i')-a-d_{\pi(i)}$ lies in $A/\im{1-f}$, for each $i\in I'$, and therefore
$(d_{j}:j\in I)$ is a multitransversal of $A/\im{1-f}$, hence balanced.
\end{proof}

Theorem \ref{thm:iso} now simplifies a lot when considering affine quandles.

\begin{corollary}\label{cor:isothm_affineX}
 Let $\ext{A,f,(d_i:i\in I)}$ and $\ext{A',f',(d_i':i\in I')}$ be two indecomposable extensions
 such that $\bar d$ and $\bar d'$ are balanced.
 The extensions are isomorphic if and only if there is an isomorphism $\psi:A\to A'$ such that
\begin{itemize}
	\item[(E1)] $\psi f=f'\psi$,
	\item[(E2')] the multiplicities of~$\bar d$, $\bar d'$ are equal.
 \end{itemize}
\end{corollary}

\begin{proof}
 We need to prove that (E2') holds if and only if there exist $\pi$ and $a$ satisfying (E2). 
 
 ($\Leftarrow$) Choose a coset of $\im{1-f}$ and consider the subset $J\subseteq I$ of all indices $j$ such that $d_j$ belongs to this coset.
 Then (E2) implies that all $d'_{\pi(j)}$, $j\in J$, belong to the same coset of $\im{1-f'}$, hence the multiplicity of $\bar d'$ must be greater or equal. A reverse argument shows that they are equal.

 ($\Rightarrow$) Let $a=0$ and take $\pi$ such that $d_i\in u+\im{1-f}$ if and only if $d'_{\pi(i)}\in\psi(u)+\im{1-f'}$. This is possible since $\bar d$, $\bar d'$ are multitransversals of block systems with the same numbers of blocks and the same multiplicity. Now, $\psi(d_i)\in\psi(u+\im{1-f})=\psi(u)+\im{1-f'}$, and thus $\psi(d_i)-d'_{\pi(i)}\in\im{1-f'}$.
\end{proof}

As a byproduct, we obtain Hou's isomorphism theorem for affine quandles.

\begin{corollary}\cite[Theorem 3.1]{Hou-aut}\label{cor:isothm_affine}
Two affine quandles $\aff{A,f}$, $\aff{B,g}$ are isomorphic if and only if there exists an isomorphism $\psi:\im{1-f}\to\im{1-g}$ such that $\psi f(u)=g\psi(u)$ for every $u\in\im{1-f}$, and
\[|\Ker{1-f}/\Ker{1-f}\cap\im{1-f}|=|\Ker{1-g}/\Ker{1-g}\cap\im{1-g}|.\]
\end{corollary}

\begin{proof}
According to Example \ref{ex:affine_ext}, $\aff{A,f}\simeq\ext{(1-f)(A),f,((1-f)(t):t\in T)}$ where $T$ is a transversal of $A/\im{1-f}$, and $\aff{B,g}\simeq\ext{(1-g)(B),g,((1-b)(s):s\in S)}$ where $S$ is a transversal of $B/\im{1-g}$.
 Now, Corollary \ref{cor:isothm_affineX} applies: the conditions on $\psi$ are identical, and according to Lemma~\ref{l:mtr2}, $|\Ker{1-f}/\Ker{1-f}\cap\im{1-f}|$ equals the multiplicity of $\{\{(1-f)(t):t\in T\}\}$.
\end{proof}

\section{Enumeration}\label{sec:enum}

First, we outline an enumeration procedure for quasi-affine quandles of given order $n$. Theorem \ref{th:Isom_Ab} suggests to start with a choice of
\begin{itemize}
	\item an index set $I$ of order $k$ dividing $n$ (then $k$ will be the number of orbits),
	\item an abelian group $A$ of order $n/k$ up to isomorphism (then $n/k$ will be the orbit size),
	\item its automorphism $f$ up to conjugacy (the orbit subquandles will be $\simeq\aff{A,f}$).
\end{itemize}	
The number of quandles of the form $\ext{A,f,(d_1,\dots,d_k)}$ will be denoted by $\varepsilon(A,f,k)$, and we will often shortcut $\varepsilon_{n,k}=\varepsilon(\Z_n,1,k)$.
In some cases, the enumeration is easy to do with ad hoc arguments.

\begin{proposition}\label{p:enum_p}
There are exactly $p-1$ quasi-affine quandles of order $p$, $p$ prime, up to isomorphism. All of them are affine.
\end{proposition}

\begin{proof}
We will follow the procedure outlined above.
We can choose $k=1$ or $k=p$. 
For $k=1$, we have $A=\Z_p$ and $f(x)=ax$ for some $a\in\{1,\ldots,p-1\}$. The extension is indecomposable if and only if $a\neq1$. The choice of $\bar d=(d_1)$ is irrelevant and we obtain $p-2$ connected affine quandles.
For $k=p$, we have $A=\{0\}$ the trivial group, $f=1$, and there is only one choice $\bar d=(0,\dots,0)$. We obtain one affine (projection) quandle.
\end{proof}

\begin{proposition}\label{p:enum_pq}
Let $p,q$ be distinct primes. Up to isomorphism, there are exactly
\begin{enumerate}
	\item $2p^2-2p-2+\varepsilon_{p,p}$ quasi-affine quandles of order $p^2$,
	\item $pq-p-q+1+\varepsilon_{p,q}+\varepsilon_{q,p}$ quasi-affine quandles of order $pq$.
\end{enumerate}
\end{proposition}

\begin{proof}
We will follow the procedure outlined above.

(1) For $k=1$, we have $|A|=p^2$ and we obtain precisely the connected affine quandles of order $p^2$. It was determined by Hou \cite{Hou} that there are $2p^2-3p-1$ such quandles, up to isomorphism.

For $k=p$, we have $A=\Z_p$. \emph{Subcase 1:} $f(x)=ax$ for $a\in\{2,\ldots,p-1\}$. Then the orbits are latin subquandles, and according to Corollary \ref{cor:latin_orbits}, $Q=\aff{A,f}\times\mathrm{Proj}(p)$; this way, we obtain $p-2$ affine quandles. 
\emph{Subcase 2:} $f=1$. This case contributes $\varepsilon_{p,p}$ quandles.

For $k=p^2$, we have $A=\{0\}$ the trivial group, $f=1$, and there is only one choice $\bar d=(0,\dots,0)$. We obtain one affine (projection) quandle.

(2) For $k=1$, we have $A=\Z_p\times\Z_q$ and we obtain precisely the products of connected affine quandles of orders $p$ and $q$, that is, $(p-2)(q-2)$ quandles.

For $k=q$, we have $A=\Z_p$. The discussion is exactly as in the second case of part (1), obtaining $p-2$ quandles of the form $Q=\aff{A,f}\times\mathrm{Proj}(q)$ and $\varepsilon_{q,p}$ quandles with $f=1$.
The case $k=p$ in analogical, with the role of $p$ and $q$ switched, contributing $q-2+\varepsilon_{p,q}$ quandles.

The case $k=pq$ is exactly as the last case of part (1).
\end{proof}

Now, we will address how to calculate the numbers $\varepsilon(A,f,k)$.
Let $Q=\ext{A,f,(d_1,\dots,d_k)}$ and $Q'=\ext{A,f,(d'_1,\dots,d'_k)}$ be two indecomposable semiregular extensions. Theorem \ref{thm:iso} says that
$Q\simeq Q'$ if and only if $\bar{d'}$ is obtained from $\bar{d}$ using the following four types of transformations:
\begin{itemize}
\item[(T1)] the sequence can be permuted;
\item[(T2)] any element of the sequence can be replaced by an element from the same coset of~$\im{1-f}$;
\item[(T3)] all elements of the sequence can be translated (simultaneously) by an element of $A$ (i.e., we apply the mapping $x\mapsto x+u$ on every element of the sequence);
\item[(T4)] all elements of the sequence can be mapped (simultaneously) by an automorphism of $A$ which commutes with $f$, i.e., by an element of the centralizer $C_{\aut A}(f)$.
\end{itemize}

An application of the following proposition allows to reduce many enumeration problems to the case $f=1$. 
Observe that, for $\psi\in C_{\aut A}(f)$, the mapping $\psi/\im{1-f}$ defined by $a+\im{1-f}\mapsto \psi(a)+\im{1-f}$ is a well defined automorphism of $A/\im{1-f}$.

\begin{proposition}\label{p:enum_red}
For every $A,f,k$, \[\varepsilon(A,f,k)\leq\varepsilon(A/\im{1-f},1,k).\]
Moreover, equality holds if \[\aut{A/\im{1-f}}=\{\psi/\im{1-f}:\ \psi\in C_{\aut A}(f)\}.\]
In particular, it holds for every $A$ cyclic.
\end{proposition}

\begin{proof}
First, consider $\varepsilon(A,f,k)$. Choose a transversal $T$ to $A/\im{1-f}$. Due to (T2), we can assume that all $d_i\in T$. Now (T1) say that we count such tuples $\bar d$ up to the order of their entries.
By (T3) we count up to translation by an element of $A$ modulo $\im{1-f}$, i.e., applying $x\mapsto x+u$, if $x+u\not\in T$, it is replaced by a representative of its coset.
And by (T4) we count up to  application of an automorphism $\psi\in C_{\aut A}(f)$ modulo $\im{1-f}$, i.e., if $\psi(x)\not\in T$, it is replaced by a representative of its coset.

Next, consider $\varepsilon(A/\im{1-f},1,k)$. Here (T2) is trivial, and (T1), (T3), (T4) says that we count tuples $\bar d$ up to the order of their entries, translation by an element of $A/\im{1-f}$ and application of an automorphism $\psi\in\aut{A/\im{1-f}}$.

Indeed, $\aut{A/\im{1-f}}\supseteq\{\psi/\im{1-f}:\ \psi\in C_{\aut A}(f)\}$ for any $A,f$. Since bigger groups have less orbits, we obtain the inequality. If the two sets are equal, we obtain equality. For cyclic groups, the two sets are always equal: given an automorphism $\rho$ of $A/\im{1-f}$, choose $u\in A$ such that $\rho(1+\im{1-f})=u+\im{1-f}$ and define $\psi(x)=ux$. Indeed, $\psi\in C_{\aut A}(f)=\aut A$ and $\psi/\im{1-f}=\rho$.
\end{proof}

In particular, if $1-f$ is onto, we obtain that $\varepsilon(A,f,k)=1$. This is in accordance with Corollary \ref{cor:latin_orbits} which says that this one quandle is the direct product $\aff{A,f}\times\mathrm{Proj}(k)$.

Automorphisms of $\aut{A/\im{1-f}}$ are not always quotients of automorphisms from $C_{\aut A}(f)$. For example, if 
\[A=\Z_2^3 \text{ and }f=\left(\begin{smallmatrix}1&0&1\\0&1&0\\0&0&1\end{smallmatrix}\right),\]
then $\im{1-f}=\langle(1,0,0)\rangle$. Hence $|\aut{A/\im{1-f}}|=6$, but the centralizer $C_{\aut A}(f)$ is the subgroup of all upper triangular matrices with 1s on the diagonal and $|\{g/\im{1-f}:g\in C_{\aut A}(f)\}|=2$.

In the rest of the paper, we will address the case $f=1$. The direct approach outlined above is good if $k$ is very small.

\begin{proposition}\label{p:enum_2}
$\varepsilon(A,1,2)=1$ if $A$ is cyclic, and $\varepsilon(A,1,2)=0$ otherwise.
\end{proposition}

\begin{proof}
Let $Q=\ext{A,1,(d_1,d_2)}$. The extension is indecomposable if and only if $A=\langle d_1-d_2\rangle$. This is not possible unless $A$ is cyclic.
Now, due to (T1) and (T3), we can assume that $d_1=0$ and $d_2=u$ is a generator of $A$. But then $x\mapsto ux$ is an automorphism of $A$, and thus $\ext{A,f,(0,u)}\simeq\ext{A,f,(0,1)}$ by (T4).
\end{proof}

%\begin{example}\label{Ex:aux}
%We show that $\varepsilon(\Z_2^2,f,2)=\varepsilon(\Z_2^2,f,3)=1$ where $f$ is an automorphism of order 2. Without loss of generality, let $\im{1-f}=\langle(1,1)\rangle$. Consider an indecomposable extension $\ext{\Z_2^2,f,\bar d}$. Due to (T2), we can assume that $d_i=(0,0)$ or $d_i=(1,0)$, and indecomposability requires that both values must be present. The order does not matter and by a translation by $(1,0)$, we can flip the values. For both $k=2,3$, there is only a unique choice of $\bar d$.
%\end{example}

As $k$ grows, an analysis as in the previous proof becomes infeasible (already for $\varepsilon_{p,3}$, this method would result in a tedious case study). 
To proceed further, we need a better approach. 

Let us start calculating $\varepsilon(A,1,k)$. By (T1), the order of $\bar d$ is irrelevant, hence we can record only the numbers $c_a$, $a\in A$, of occurrences of elements of $A$ in the tuple $\bar d$. Clearly, the sum $\sum_{a\in A} c_a$ must be equal to $k$, the length of $\bar d$.
Let \[C(A,k)=\{(c_a:a\in A):\ \sum_{a\in A} c_a=k \text{ and the corresponding extension is indecomposable}\}.\]
If $A$ is cyclic, the corresponding extension is indecomposable if and only if $\bar d$ is not constant, that is, if and only if $c_a\neq k$ for every $a\in A$. 
Using a standard combinatorial trick \cite[Section 3.3]{MN}, we see that \[|C(\Z_n,k)|=\binom{n+k-1}{n-1}-n.\]
(If $A$ is not cyclic, the condition is more complicated; the discussion is omitted here.)

\begin{lemma}\label{l:enum p}
The number $\varepsilon(A,1,k)$ is equal to the number of orbits of the holomorph $A\rtimes\aut A$ acting on $C(A,k)$ by permuting indices, 
i.e., under the action $(u,\psi)(c_a:a\in A)=(c_{u+\psi(a)}:a\in A)$. 
\end{lemma}

\begin{proof}
The transformation (T2) is trivial if $f=1$.
The transformation (T3) corresponds to the action of the mapping $x\mapsto x+u$ on the indices of the sequence $\bar c$.
The transformation (T4) corresponds to the action of the mapping $\psi\in\aut A=C_{\aut A}(1)$ on the indices of $\bar c$. 
Therefore, two sequences $\bar d,\bar d'$ yield isomorphic quandles if and only if the corresponding sequences $\bar c,\bar c'$ are in the same orbit of the action of the holomorph $A\rtimes\aut{A}$ on $C(A,k)$.
\end{proof}

To calculate the number of orbits, we can use Burnside's orbit counting lemma:
\[\varepsilon(A,1,k)=\frac{1}{|A|\cdot|\aut A|}\cdot\sum_{g\in A\rtimes\aut A} \mathrm{Fix}(g),\]
where $\mathrm{Fix}(g)$ denotes the number of sequences from $C(A,k)$ fixed by $g$.

\begin{proposition}\label{p:enum_eps}
For every $k$, we have
\begin{align*}
	\varepsilon_{2,k}&=\left\lfloor\frac k2\right\rfloor,\\ 
	\varepsilon_{3,k}&=\frac1{12}\left(k^2+6k-4+\xi_k\right),\text{ where } 
	\xi_k= \begin{cases} 4\text{ if } k\equiv0\pmod 6,\\ 1\text{ if } k\equiv3\pmod 6,\\ 0\text{ if } k\equiv2,4\pmod 6,\\ -3\text{ if } k\equiv1,5\pmod 6, \end{cases} \\
	\varepsilon_{4,3}&=2,\quad \varepsilon_{5,3}=2,\quad \varepsilon(\Z_2^2,1,3)=1.
\end{align*}
%\begin{enumerate}
%	\item $\varepsilon_{2,k}=\lfloor\frac k2\rfloor$, 
%	\item $\varepsilon_{3,k}=(k^2+6k-4+\xi_k)/12$, where $\xi_k=4$ if $k\equiv0\pmod 6$, $\xi_k=1$ if $k\equiv3\pmod 6$, $\xi_k=0$ if $k\equiv2,4\pmod 6$, $\xi_k=-3$ if $k\equiv1,5\pmod 6$,
%	\item $\varepsilon_{4,3}=2$, $\varepsilon_{5,3}=2$, $\varepsilon(\Z_2^2,1,3)=1$.
%\end{enumerate}
\end{proposition}

\begin{proof}
For $\varepsilon_{2,k}$, we consider pairs $(c_0,c_1)$ such that $c_0+c_1=k$, $c_i\neq k$. The holomorph acts on indices as the permutation group $\langle(0\ 1)\rangle$, hence the orbits can be uniquely represented by the pairs with $c_0\leq c_1$, and thus there are $\lfloor\frac k2\rfloor$ orbits. 

For $\varepsilon_{3,k}$, we have $C(\Z_3,k)=\{(c_0,c_1,c_2):\ c_0+c_1+c_2=k,\ c_i\neq k\}$, and the holomorph acts on indices as the group $G=\langle(0\ 1\ 2),(0\ 1)\rangle=S_3$, the symmetric group on three elements. Applying Burnside's formula, we obtain 
\[\varepsilon_{3,k}=\frac16\cdot\left(\binom{k+2}2-3+2\cdot\zeta_k+3\cdot\left\lfloor\frac k2\right\rfloor\right),\]
where $\zeta_k$ counts the number of triples fixed by a 3-cycle, that is, $\zeta_k=1$ if $k\equiv0\pmod 3$, and $\zeta_k=0$ otherwise.
The last term, $\lfloor\frac k2\rfloor$, is the number of triples fixed by a 2-cycle: such triples must have two entries equal to a number $\leq k/2$.
Replacing $\lfloor\frac k2\rfloor$ by $\frac k2$ plus 0 or $-\frac12$, depending on parity of $k$, we obtain the expression stated above.

For $\varepsilon_{4,k}$, we have $C(\Z_4,k)=\{(c_0,c_1,c_2,c_3):\ c_0+c_1+c_2+c_3=k,\ c_0+c_2\neq k,\ c_1+c_3\neq k\}$ (here, the indecomposability condition requires that the differences $d_i-d_j$ contain 1 or 3). The holomorph acts as the group $\langle(0\ 1\ 2\ 3),(1\ 3)\rangle=D_8$, the dihedral group on 8 elements.
For $k=3$, the only admissible quadruples contain 2,1,0,0 or 1,1,1,0 (in an arbitrary order), and thus the Burnside's formula gives $(12+2\cdot2)/8=2$, where the first term comes from $g=1$, the identity, and the second term comes from the two transpositions.

For $\varepsilon_{5,k}$, we have $C(\Z_5,k)=\{(c_0,...,c_4):\ \sum c_i=k,\ c_i\neq k\}$ and the holomorph acts as the group $\langle(0\ 1\ 2\ 3\ 4),(1\ 4)(2\ 3)\rangle$.
For $k=3$, the only admissible quadruples contain 2,1,0,0,0 or 1,1,1,0,0, and thus the Burnside's formula gives $(30+5\cdot2)/20=2$, where the first term comes from $g=1$ and the second one from the five permutations with two 2-cycles.

For $\varepsilon(\Z_2^2,1,3)$, the set $C(\Z_2^2,3)$ contains only four quadruples consisting of 1,1,1,0, and the holomorph acts as the symmetric group on four elements, hence the Burnside's formula gives $(4+6\cdot2+8\cdot1)/24=1$, where the first term comes from $g=1$, the second from transpositions and the third from 3-cycles.
\end{proof}

\begin{corollary}\label{Iso2p}
There are, up to isomorphism, $(3p-1)/2$ quasi-affine quandles of order $2p$, for a prime $p>2$.
\end{corollary}

In Table \ref{Fig:count_affine}, we display the number of quasi-affine, affine, and affine latin quandles of orders up to 15 (note that connected medial quandles are affine, and if finite then also latin, see \cite[Section 7]{HSV}). For prime orders, see Proposition \ref{p:enum_p}. For orders $2p$ we use Corollary~\ref{Iso2p}.
For orders $4,9,15$, combine Propositions \ref{p:enum_pq} and \ref{p:enum_eps}. To complete orders 8 and 12, we need a more detailed analysis in the case when $1-f$ is neither an automorphism, nor zero, see Examples \ref{Ex:8} and \ref{Ex:12}. The numbers of affine quandles come from \cite{Hou}.

\begin{table}[h]
$$\begin{array}{|r|rrrrrrrrrrrrrrr|}\hline
n 															 &1&2&3&4&5&6&7&8& 9&10&11&12&13&14&15\\\hline
\text{quasi-affine}              &1&1&2&3&4&4&6&9&12& 7&10&17&12&10&14 \\
\text{affine}   								 &1&1&2&3&4&2&6&7&11& 4&10& 6&12& 6& 8 \\
\text{affine latin}   					 &1&0&1&1&3&0&5&2& 8& 0& 9& 1&11& 0& 3 \\\hline
\end{array}$$
\caption{The number of quasi-affine, affine, and affine latin quandles of order $n$, up to isomorphism.}
\label{Fig:count_affine}
\end{table}

\begin{example}\label{Ex:8}
We will calculate the number of quasi-affine quandles of order 8 up to isomorphism, following the procedure outlined at the beginning of this section.
For $k=1$, we have $|A|=8$ and we obtain precisely the connected affine quandles of order $8$. It is well known (cf. \cite{Hou}) that there are two of them.
For $k=2$, we have $A=\Z_4$ or $A=\Z_2^2$. Let $\{1,f,g\}$ be representatives of conjugacy classes of $\aut{\Z_2^2}$, with $f$ of order 2 and $g$ of order 3.
The contribution is 
\[\varepsilon(\Z_4,1,2)+\varepsilon(\Z_4,-1,2)+\varepsilon(\Z_2^2,1,2)+\varepsilon(\Z_2^2,f,2)+\varepsilon(\Z_2^2,g,2)=
1+\varepsilon_{2,2}+0+\varepsilon_{2,2}+1=4,\]
using Propositions \ref{p:enum_2}, \ref{p:enum_red}, \ref{p:enum_2}, \ref{p:enum_red}, and \ref{cor:latin_orbits}, respectively.
For $k=4$, we have $A=\Z_2$ and the contribution is $\varepsilon_{2,4}=2$.
For $k=8$, we obtain one projection quandle of order 8.
The total is $2+4+2+1=9$ quasi-affine quandles of order 8 up to isomorphism.
\end{example}

\begin{example}\label{Ex:12}
We will calculate the number of quasi-affine quandles of order 12 up to isomorphism, following the procedure outlined at the beginning of this section.
For $k=1$, we obtain precisely the connected affine quandles of order $12$, which decompose to a direct product of one of order 4 and one of order 3. There is precisely one such pair.
For $k=2$, we have $A=\Z_6$, contributing $\varepsilon(\Z_6,1,2)+\varepsilon(\Z_6,-1,2)=\varepsilon_{6,2}+\varepsilon_{2,2}=2$ quandles, using Propositions \ref{p:enum_2} and \ref{p:enum_red}, respectively.
For $k=3$, we have $A=\Z_4$ or $A=\Z_2^2$. Let $\{1,f,g\}$ be representatives of conjugacy classes of $\aut{\Z_2^2}$, with $f$ of order 2 and $g$ of order 3.
The contribution is 
\[\varepsilon(\Z_4,1,3)+\varepsilon(\Z_4,-1,3)+\varepsilon(\Z_2^2,1,3)+\varepsilon(\Z_2^2,f,3)+\varepsilon(\Z_2^2,g,3)=
\varepsilon_{4,3}+\varepsilon_{2,2}+1+\varepsilon_{2,3}+1=6,\]
using Propositions \ref{p:enum_eps}, \ref{p:enum_red}, \ref{p:enum_eps}, \ref{p:enum_red}, and \ref{cor:latin_orbits}, respectively.
For $k=4$, we have $A=\Z_3$, and the contribution is $\varepsilon(\Z_3,1,4)+\varepsilon(\Z_3,-1,4)=3+1=4$.
For $k=6$, we have $A=\Z_2$, and the contribution is $\varepsilon_{2,6}=3$.
For $k=12$, we obtain one projection quandle of order 12.
The total is $1+2+6+4+3+1=17$ quasi-affine quandles of order 12 up to isomorphism.
\end{example}

\end{document}